\documentclass[10pt,reqno]{amsart}
\usepackage{amsmath}
\usepackage{amssymb}
\usepackage{amsthm}
\usepackage{eepic,epic}
\usepackage{epsfig}
\usepackage{graphicx}

\newlength{\defbaselineskip}
\newcommand{\setlinespacing}[1]%
           {\setlength{\baselineskip}{#1 \defbaselineskip}}

\numberwithin{equation}{section}

\newtheorem{thm}{Theorem}[section]
\newtheorem{cor}[thm]{Corollary}
\newtheorem{lem}[thm]{Lemma}
\newtheorem{prop}[thm]{Proposition}

\theoremstyle{definition}

\theoremstyle{remark}
\newtheorem{rem}[thm]{Remark}
\numberwithin{equation}{section}

\begin{document}

\title[Morawetz estimates with time-dependent weights]
{On Morawetz estimates with time-dependent weights for the Klein-Gordon equation}

\author{Jungkwon Kim, Hyeongjin Lee, Ihyeok Seo and Jihyeon Seok}

\thanks{This research was supported by NRF-2019R1F1A1061316.}


\address{Department of Mathematics, Sungkyunkwan University, Suwon 16419, Republic of Korea}
\email{kimjk809@skku.edu}
\email{hjinlee@skku.edu}
\email{ihseo@skku.edu}
\email{jseok@skku.edu}

\begin{abstract}
We obtain some new Morawetz estimates for the Klein-Gordon flow of the form
\begin{equation*}
\big\||\nabla|^{\sigma} e^{it \sqrt{1-\Delta}}f \big\|_{L^2_{x,t}(|(x,t)|^{-\alpha})} \lesssim \|f\|_{H^s}
\end{equation*}
where $\sigma,s\geq0$ and $\alpha>0$.
The conventional approaches to Morawetz estimates with $|x|^{-\alpha}$
are no longer available in the case of time-dependent weights $|(x,t)|^{-\alpha}$.
Here we instead apply the Littlewood-Paley theory with Muckenhoupt $A_2$ weights to frequency localized estimates thereof
that are obtained by making use of the bilinear interpolation between their bilinear form estimates
which need to carefully analyze some relevant oscillatory integrals
according to the different scaling of $\sqrt{1-\Delta}$ for low and high frequencies.
\end{abstract}

\maketitle

\section{Introduction}\label{sec1}

In this paper we are concerned with weighted $L^2$ estimates for the Klein-Gordon flow $e^{it\sqrt{1-\Delta}}$
which gives a solution formula (see \eqref{sol}) for the Klein-Gordon equation $\partial_{t}^2 u-\Delta u+u = 0$.
Let us first consider
\begin{equation}\label{mor}
\big\||\nabla|^{\sigma} e^{it\sqrt{1-\Delta}}f\big\|_{L_{x,t}^2(|x|^{-\alpha})} \lesssim \|f\|_{H^{1/2}},
\end{equation}
where $\sigma\in\mathbb{R}$ and $\alpha>0$.
This type of estimates was firstly initiated by Morawetz \cite{M} and is sometimes called Morawetz estimates.
Especially, the case $\sigma>0$ in the setting \eqref{mor} is indicative of a smoothing effect.

Let $n \ge 2$.
In \cite{OR}, Ozawa and Rogers showed that \eqref{mor} holds for $-n/2 +1 <\sigma<1/2$ and $\alpha=2-2\sigma$.
Here, the region of $(\alpha,\sigma)$ when $\sigma>0$ is given by the open segment $(A,B)$ (see Figure \ref{fig1} below).
Recently, it was shown by D'Ancona \cite{D'A2} (see also \cite{LSS}) that the flow can have much smoothness locally as
\begin{equation}\label{LS0}
\sup_{R>0} \frac{1}{R} \int_{|x|<R} \int_{-\infty}^\infty
\big||\nabla|^{1/2}e^{it\sqrt{1-\Delta}}f\big|^2 dtdx \lesssim \|f\|_{H^{1/2}}^2.
\end{equation}
At this point, it should be noted that this local smoothing estimate can be written in terms of a weighted $L^2$ setting as in \eqref{mor}.
Indeed, if $\rho$  is any function such that $\sum_{j\in\mathbb{Z}}\|\rho\|_{L^\infty(|x|\sim2^j)}^2<\infty$,
one has
$$\big\|\rho|x|^{-1/2}h\big\|_{L_{x,t}^2}\lesssim\sup_{R>0} \frac{1}{R} \int_{|x|<R} \int_{-\infty}^\infty|h|^2 dtdx,$$
and then \eqref{LS0} implies a weaker estimate,
\begin{equation}\label{this}
\big\||\nabla|^{1/2}e^{it\sqrt{1-\Delta}}f\big\|_{L_{x,t}^2(\rho^2|x|^{-1})} \lesssim \|f\|_{H^{1/2}}^2.
\end{equation}
A typical example of such $\rho$ is given by
$$\rho=(1+(\log|x|)^2)^{-\frac12(\frac12+\varepsilon)},\quad \varepsilon>0,$$
as mentioned in \cite{D'A2}.
But this case does not cover the point $A$ because the weight $\rho^2|x|^{-1}$ in \eqref{this}
decays faster than $|x|^{-1}$.

The smoothing effect in the setting \eqref{mor} can occur due to the decay of the weight $|x|^{-\alpha}$ in the spatial direction.
As shown in Figure \ref{fig1}, the smoothing factor $\sigma$ is indeed related with the decay factor $\alpha$.
At this point, we naturally further ask how much regularity we can expect when considering decay
not in the spatial direction but in the space-time direction like $|(x,t)|^{-\alpha}$.
In other words, we expect a smoothing effect above the line through the points $A,B$ in this new setting (see Remark \ref{rem} below).
Our result is the following theorem:

\begin{thm}\label{thm}
Let  $n \ge 1$ and $\sigma\geq0$.
Then we have
	\begin{equation}\label{S}
	\big\||\nabla|^{\sigma} e^{it \sqrt{1-\Delta}}f \big\|_{L^2_{x,t}(|(x,t)|^{-\alpha})} \lesssim \|f\|_{H^s(\mathbb{R}^n)}
	\end{equation}
if $\frac{n+2}2\leq\alpha<n+1$, $s> \frac{n\alpha}{4(n+1)}$ and $s-\frac n4<\sigma <s- \frac{n\alpha}{4(n+1)}$.
\end{thm}

\begin{rem}\label{rem}
Thus we have a smoothing effect in the setting \eqref{S} even if $n=1$.
Compared to \eqref{mor}, the estimate \eqref{S} when $s=1/2$ holds whenever $\frac{n+2}2\leq\alpha<\min\{n+1,\frac{2(n+1)}{n}\}$
and $\frac12-\frac n4<\sigma <\frac12- \frac{n\alpha}{4(n+1)}$ for $n\leq3$.
Here, $(\alpha,\sigma)$ lies in some region inside the triangle with vertices $C,B,D$, as expected.
\end{rem}

\begin{figure}
	\centering	\includegraphics[width=0.7\textwidth]{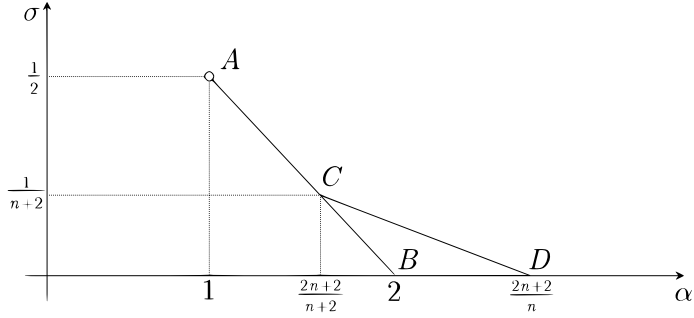}
	\label{fig1}
   \caption{The range of $(\alpha,\sigma)$ for \eqref{mor} and \eqref{S} with $s=1/2$.}
\end{figure}


In general, there are two approaches to smoothing estimates with time-independent weights.
One is spectral methods based on resolvent estimates, starting from Kato's work \cite{K2},
and the other is to use Fourier restriction estimates from harmonic analysis (see e.g. \cite{BBCRV,S} and references therein).
A completely different approach was also recently
developed by Ruzhansky and Sugimoto \cite{RS} using canonical transforms.

However, these are no longer available in the case of time-dependent weights.
Here we instead combine the theory of dispersive estimates ($TT^*$ argument and bilinear approach)
and the Littlewood-Paley theory with Muckenhoupt $A_p$ weights.
This allows us to take advantage of localization on the frequency space.
This more fruitful and direct approach has been developed by Koh and the third author \cite{KS,KS2,KS3,KS4} for the Schr\"odinger and wave flows.
But here we cannot take advantage of scaling consideration any more.
In general, the Klein-Gordon flow $e^{it\sqrt{1-\Delta}}$ is more complicated by the different scaling of $\sqrt{1-\Delta}$
for low and high frequencies.
Indeed, it behaves like the wave flow at high frequency and the Schr\"odinger flow at low frequency.
Hence we need to modify delicately the approach to make it applicable to the Klein-Gordon flow
according to low and high frequencies.

The solution $u$ of the Klein-Gordon equation $\partial_{t}^2 u-\Delta u+u = 0$ with initial data $u(x,0)=f(x)$ and
$\partial_{t} u(x,0)=g(x)$ is given by
\begin{equation}\label{sol}
 u(x,t) = \cos(t\sqrt{1-\Delta}) f + \frac{\sin(t\sqrt{1-\Delta})}{\sqrt{1-\Delta}}g.
\end{equation}
Thus the above theorem for the flow $e^{it\sqrt{1-\Delta}}$ implies immediately the following corollary for \eqref{sol}.

\begin{cor}
Under the same conditions as in Theorem \ref{thm}, we have
	\begin{equation*}
	\big\||\nabla|^{\sigma} u \big\|_{L^2_{x,t}(|(x,t)|^{-\alpha})} \lesssim \|f\|_{H^s} + \| g \|_{H^{s-1}}
	\end{equation*}
with the Cauchy data $(f,g) \in H^{s} \times H^{s-1}$.
\end{cor}

When thinking of possible applications to Klein-Gordon equations with potential or nonlinearity,
some inhomogeneous estimates such as
	\begin{equation*}
	\bigg\|\int_0^t\frac{e^{i(t-s) \sqrt{1-\Delta}}}{\sqrt{1-\Delta}}F(\cdot,s) \bigg\|_{L^2_{x,t}(|(x,t)|^{-\alpha})} \lesssim \|F\|_{L^2_{x,t}(|(x,t)|^{\alpha})}
	\end{equation*}
are useful in some cases. A similar argument as used for \eqref{S} is available for obtaining such inhomogeneous estimates.
But this is not a purpose of the present paper. We instead refer the reader to \cite{KS} for details.

\

\noindent\textit{Outline of paper.}
In Section \ref{sec2} we prove the estimate \eqref{S} in Theorem \ref{thm}
applying the Littlewood-Paley theory with Muckenhoupt $A_2$ weights to its frequency localized estimates in Proposition \ref{FLE}.
These localized estimates are obtained in Section \ref{sec3} by utilizing the $TT^*$ argument and then making use of the bilinear interpolation
between their bilinear form estimates (see, for example, Proposition \ref{local}).
We carry out this process by dividing cases into two parts according to the different scaling of $\sqrt{1-\Delta}$
for low (\eqref{LE}) and high (\eqref{LE2}) frequencies.
The proof of \eqref{LE2} is much more delicate.
We generate further localizations in time and space and need to analyze some relevant oscillatory integrals more carefully
(see Lemma \ref{osci})
based on the stationary phase lemma, Lemma \ref{stationary_phase}.

Throughout this paper, the letter $C$ stands for a positive constant which may be different at each occurrence.
We also denote $A\lesssim B$ to mean $A\leq CB$
with unspecified constants $C>0$.

\section{Proof of Theorem 1.1}\label{sec2}

In this section we prove Theorem \ref{thm} by making use of the Littlewood-Paley theory on weighted $L^2$ spaces.

Let us first recall that a weight\footnote{It is a locally integrable function that is allowed to be zero or
infinite only on a set of Lebesgue measure zero.} $w : \mathbb{R}^{n} \rightarrow [0, \infty]$ is said to be
in the Muckenhoupt $A_2(\mathbb{R}^{n})$ class if there is a constant $C_{A_2}$ such that
	$$ \sup_{Q \, \textnormal{cubes in } \mathbb{R}^{n}} \left( \frac{1}{|Q|}\int_{Q} w(x) dx \right)\left( \frac{1}{|Q|}\int_{Q}w(x)^{-1}dx \right)<C_{A_2}.$$
(See e.g. \cite{G} for details.)
We then observe that the weight $|(x,t)|^{-\alpha}$ for $-(n+1) < \alpha < n+1$ is in $A_2(\mathbb{R}^n)$ with $C_{A_2}$
which is uniform in almost every $t \in \mathbb{R}$.
Indeed, $|(x,t)|^{-\alpha}$ is in $A_2(\mathbb{R}^{n+1})$ for $-(n+1) < \alpha < n+1$.
That is to say, there is a constant $C_{A_2}$ such that
\begin{equation}\label{spw}
\sup_{\widetilde{Q} \, \textnormal{cubes in } \mathbb{R}^{n+1}} \left( \frac{1}{|\widetilde{Q}|}\int_{\widetilde{Q}} |(x,t)|^{-\alpha} dxdt \right)\left( \frac{1}{|\widetilde{Q}|}\int_{\widetilde{Q}}|(x,t)|^{\alpha}dxdt \right)<C_{A_2}.
\end{equation}
For any cube $Q \subset \mathbb{R}^n$ and any interval $I \subset \mathbb{R}$,
by applying \eqref{spw} to $\widetilde{Q} = Q \times I$,
we conclude that
$$
\sup_{Q \in \mathbb{R}^{n},\, I \in \mathbb{R}} \left( \frac{1}{| I|} \int_{I}  \frac{1}{|Q |} \int_{Q} |(x,t)|^{-\alpha} dxdt \right) \left( \frac{1}{| I|} \int_{I}  \frac{1}{|Q|} \int_{Q} |(x,t)|^{\alpha}dx dt\right)<C_{A_2}.
$$
Since $|(x,t)|^{\pm\alpha}$ are $L_{\textrm{loc}}^1 (\mathbb{R}^n)$ functions for $-(n+1)<\alpha < (n+1)$,
by Lebesgue's differentiation theorem,
we can shrink $|I|$ to zero so that $|(x,t)|^{-\alpha}\in A_2(\mathbb{R}^n)$ uniformly in almost every $t$.

Now we are in a good light that the Littlewood-Paley theory (see Theorem 1 in \cite{Ku}) on weighted $L^2$ space with $A_2$ weight
is applied to get
\begin{align}\label{LP}
\nonumber\big\||\nabla|^{\sigma} e^{it\sqrt{1-\Delta}}\big(\sum_{j>0}P_jf&\big)\big\|_{L^2_{x,t}(|(x,t)|^{-\alpha})}^2
=\int_{\mathbb{R}} \big\||\nabla|^{\sigma} e^{it\sqrt{1-\Delta}}\big(\sum_{j>0}P_jf\big)\big\|_{L^2_{x}(|(x,t)|^{-\alpha})}^2 dt \nonumber\\
&\lesssim \int_{\mathbb{R}} \bigg\| \Big(\sum_{k\in\mathbb{Z}} \Big|P_k|\nabla|^{\sigma} e^{it\sqrt{1-\Delta}}\big(\sum_{j>0}P_jf\big)\Big|^2\Big)^{1/2}\bigg\|_{L^2_{x}(|(x,t)|^{-\alpha})}^2 dt \nonumber\\
&\lesssim \int_{\mathbb{R}} \bigg\| \Big(\sum_{k\geq0} \Big|P_k|\nabla|^{\sigma} e^{it\sqrt{1-\Delta}}\big(\sum_{|j-k|\leq1}P_jf\big)\Big|^2\Big)^{1/2}\bigg\|_{L^2_{x}(|(x,t)|^{-\alpha})}^2 dt \nonumber\\
&= C \sum_{k\geq0} \big\| |\nabla|^{\sigma} e^{it\sqrt{1-\Delta}}P_k\big(\sum_{|j-k|\leq1}P_jf\big)\big\|_{L^2_{x,t}(|(x,t)|^{-\alpha})}^2.
\end{align}
Here, we used the fact that $P_k P_j f = 0$ if $|j-k|\ge 2$,
where $P_k$, $ k \in \mathbb{Z}$, are Littlewood-Paley projections given by
$\widehat{P_k f}(\xi) = \phi(2^{-k}|\xi|)\widehat{f}(\xi)$
with a smooth cut-off function $\phi : \mathbb{R} \rightarrow [0,1]$ which is supported in $(1/2,2)$ and satisfies
$ \sum_{k\in\mathbb{Z}} \phi(2^k t) = 1$, $t>0$.
We also denote
$P_{\le 0}f = f-\sum_{k>0} P_k f$.
Using \eqref{LP}, we now see that
\begin{align}\label{over}
\big\||\nabla|^{\sigma} e^{it\sqrt{1-\Delta}}f\big\|_{L^2_{x,t}(|(x,t)|^{-\alpha})}
&\lesssim\big\| |\nabla|^{\sigma} e^{it\sqrt{1-\Delta}}P_{\le 0} f \big\|_{L^2_{x,t}(|(x,t)|^{-\alpha})}\\
\nonumber&+ \Big(\sum_{k\geq0} \big\| |\nabla|^{\sigma} e^{it\sqrt{1-\Delta}}P_k\big(\sum_{|j-k|\leq1}P_jf\big)\big\|_{L^2_{x,t}(|(x,t)|^{-\alpha})}^2\Big)^{1/2}.
\end{align}

To bound the right-hand side of \eqref{over}, we use the following frequency localized estimates
 for low (\eqref{LE}) and high (\eqref{LE2}) frequencies
which will be obtained in the next section.

\begin{prop}\label{FLE}
Let $n\ge1$ and $\sigma\geq0$. Assume that $1<p<\frac{2(n+1)}{n+2}$ and $ \frac{n+2}{2} \le \alpha < \frac{n+1}{p}$. Then we have
\begin{equation}\label{LE}
\big\||\nabla|^{\sigma} e^{it\sqrt{1-\Delta}}P_{\le 0} f\big\|_{L_{x,t}^2(|(x,t)|^{-\alpha})}\lesssim \|f\|_{L^2}
\end{equation}
and
\begin{equation}\label{LE2}
\big\||\nabla|^{\sigma} e^{it\sqrt{1-\Delta}}P_kf\big\|_{L_{x,t}^2(|(x,t)|^{-\alpha})}\lesssim 2^{k(\sigma+ \frac{n}{4p})} \|f\|_{L^2}
\end{equation}
for all $k\geq0$.
\end{prop}
From Proposition \ref{FLE} we indeed get
\begin{align}\label{eq12}
\nonumber\big\| |\nabla|^{\sigma} e^{it\sqrt{1-\Delta}}P_{\le 0} f \big\|_{L^2_{x,t}(|(x,t)|^{-\alpha})}
& \lesssim\big\| |\nabla|^{\sigma} e^{it\sqrt{1-\Delta}}P_{\le 0}\big(\sum_{j\leq1} P_jf\big)  \big\|_{L^2_{x,t}(|(x,t)|^{-\alpha})}\\
\nonumber&\lesssim  \|\sum_{j\leq1} P_jf\|_{L^2}\\
&\lesssim  \|P_{\leq0}f\|_{L^2}+\|P_1f\|_{L^2}
\end{align}
while
\begin{align}\label{eq22}
\sum_{k\geq0}\big\| |\nabla|^{\sigma} e^{it\sqrt{1-\Delta}}P_k\big(\sum_{|j-k|\leq1}P_jf\big)\big\|_{L^2_{x,t}(|(x,t)|^{-\alpha})}^2
&\lesssim \sum_{k\geq0}2^{2k(\sigma+ \frac{n}{4p})} \Big\|\sum_{|j-k|\le 1} P_j f\Big\|_{L^2}^2 \nonumber\\
&\lesssim   \sum_{k\geq1}2^{2k(\sigma+ \frac{n}{4p})} \|P_k f\|_{L^2}^2.
\end{align}
By combining \eqref{over}, \eqref{eq12} and \eqref{eq22}, we conclude that
\begin{align*}
\big\||\nabla|^{\sigma} e^{it\sqrt{1-\Delta}}f\big\|_{L^2_{x,t}(|(x,t)|^{-\alpha})}
&\lesssim\|P_{\le 0}f\|_{L^2}+\Big(\sum_{k\geq1}2^{2k(\sigma+ \frac{n}{4p})} \|P_k f\|_{L^2}^2\Big)^{1/2}\\
&\lesssim\|f\|_{B^{\sigma+ \frac{n}{4p}}_{2,2}}
\end{align*}
if $1<p<\frac{2(n+1)}{n+2}$ and $ \frac{n+2}{2} \le \alpha < \frac{n+1}{p}$.
Here we recall the Besov space $B^s_{r,m}$ for $s \in \mathbb{R}$ and $1 \le r,m\le \infty$ equipped with the norm
\begin{equation*}
\|f\|_{B^s_{r,m}} = \|P_{\leq0} f\|_{L^r} + \Big(\sum_{k>0} \big(2^{ks}\|P_kf\|_{L^r}\big)^m \Big)^{1/m}.
\end{equation*}
Since $B^s_{2,2}$ coincides with the inhomogeneous Sobolev space $H^s$,
by letting $s=\sigma+\frac{n}{4p}$ and eliminating the redundant exponent $p$
in the conditions on $p$, $\alpha$ and $s$,
we get \eqref{S} under the same conditions as in Theorem \ref{thm}.

\section{Frequency localized estimates: Proof of Proposition \ref{FLE}}\label{sec3}
In this section we obtain the frequency localized estimates in Proposition \ref{FLE}.
To do so, we first utilize the standard $TT^*$ argument and then make use of the bilinear interpolation
between their bilinear form estimates.
We carry out this process by dividing cases into two parts
according to the different scaling of $\sqrt{1-\Delta}$ for low (\eqref{LE}) and high (\eqref{LE2}) frequencies.

\subsection{Proof of \eqref{LE}}
First we consider an operator
$$ f \mapsto Tf :=|\nabla|^{\sigma} e^{it\sqrt{1-\Delta}}P_{\le 0} f$$
and its adjoint operator
$$ F \mapsto T^* F := \int_{-\infty}^{\infty}|\nabla|^{\sigma} e^{-is\sqrt{1-\Delta}} P_{\le 0} F(\cdot, s) ds. $$
From the standard $TT^*$ argument, \eqref{LE} is then equivalent to
\begin{equation*}\label{TT*s}
\bigg\|\int_{-\infty}^{\infty} |\nabla|^{2\sigma} e^{i(t-s)\sqrt{1-\Delta}}\,P_{\le 0}^2 F(\cdot,s)\, ds\bigg\|_{L^2_{x,t}(|(x,t)|^{-\alpha})} \lesssim  \|F\|_{L^2_{x,t}(|(x,t)|^{\alpha})}.
\end{equation*}
Again by duality, it suffices to show the following bilinear form estimate
$$
\bigg|\bigg\langle\int_{-\infty}^{\infty} |\nabla|^{2\sigma} e^{i(t-s)\sqrt{1-\Delta}}P_{\le 0}^2 F(\cdot,s) ds, G(x,t)\bigg\rangle_{L_{x,t}^2}\bigg|
\lesssim  \|F\|_{L^2_{x,t}(|(x,t)|^{\alpha})}\|G\|_{L^2_{x,t}(|(x,t)|^{\alpha})}.
$$
To show this, we first write
\begin{equation*}
\int_{-\infty}^\infty |\nabla|^{2\sigma} e^{i(t-s)\sqrt{1-\Delta}} P_{\leq0}^2 F(\cdot, s) ds = K * F,
\end{equation*}
where
\begin{equation*}
K(x,t)= \int_{\mathbb{R}^n} e^{i(x\cdot \xi +t\sqrt{1+|\xi|^2})} |\xi|^{2\sigma} \varphi(\xi)^2 d\xi
\end{equation*}
with a smooth cut-off function $\varphi : \mathbb{R}^n \rightarrow [0,1]$ supported in $\{\xi\in\mathbb{R}^n:|\xi|\leq2\}$.
Next we decompose the kernel $K$ as
\begin{equation*}
K=\sum_{j \geq 0} \psi_j K,
\end{equation*}
where $\psi_j : \mathbb{R}^{n+1} \rightarrow [0,1]$ is a smooth function which is supported in $B(0,1)$ for $j=0$
and in $B(0,2^j)\setminus B(0,2^{j-2})$ for $j \geq 1 $, such that $\sum_{j \geq 0} \psi_j = 1$. Then it is enough to show that
\begin{equation}\label{wts1}
\sum_{j \geq 0} |\langle (\psi_j K)* F , G \rangle| \leq C \| F \|_{L^2 (|(x,t)|^{\alpha})} \| G \|_{L^2 (|(x,t)|^{\alpha})}.
\end{equation}
For this, we assume for the moment that for $p>0$ and $0<\alpha< (n+1)/p$
\begin{align}
\label{bi1}
|\langle (\psi_j K)* F , G \rangle| &\leq C2^{j(\frac{n+2}{2} - \alpha p)} \| F \|_{L^2 (|(x,t)|^{\alpha p})} \| G \|_{L^2 (|(x,t)|^{\alpha p})},\\
\label{bi2}
|\langle (\psi_j K)* F , G \rangle| &\leq C2^{j(\frac{n+2}{2} - \frac{\alpha p}{2})} \| F \|_{L^2 (|(x,t)|^{\alpha p})} \| G \|_{L^2},\\
\label{bi3}
|\langle (\psi_j K)* F , G \rangle| &\leq C2^{j(\frac{n+2}{2} - \frac{\alpha p}{2})} \| F \|_{L^2} \| G \|_{L^2 (|(x,t)|^{\alpha p})},
\end{align}
and use the following bilinear interpolation lemma (see \cite{BL}, Section 3.13, Exercise 5(b)).

\begin{lem}\label{bil}
	For $i=0,1$, let $A_i , B_i , C_i$ be Banach spaces and let $T$ be a bilinear operator such that $T: A_0 \times B_0 \rightarrow C_0$, $T : A_0 \times B_1 \rightarrow C_1$, and $T : A_1 \times B_0 \rightarrow C_1$. Then one has
	\begin{equation*}
	T : (A_0,A_1)_{\theta_0, q} \times (B_0, B_1)_{\theta_1, r} \rightarrow (C_0, C_1)_{\theta,1}
	\end{equation*}
	if $0< \theta_i < \theta=\theta_0+\theta_1 <1$ and $1/q + 1/r \geq 1$ for $1 \leq q,r \leq \infty$.
Here, $(\cdot\,,\cdot)_{\theta,p}$ denotes the real interpolation functor.
\end{lem}

Indeed, let us first define the bilinear vector-valued operator $T$ by
\begin{equation*}
T(F,G) = \left\{ \langle(\psi_j K)*F, G \rangle \right\}_{j \geq 0}.
\end{equation*}
Then \eqref{wts1} is equivalent to
\begin{equation}\label{after_int}
T : L^2 (|(x,t)|^{\alpha}) \times L^2 (|(x,t)|^{\alpha}) \rightarrow \ell_1^0 (\mathbb{C}).
\end{equation}
Here, for $a \in \mathbb{R}$ and $ 1 \leq p \leq \infty$, $\ell_p^a (\mathbb{C})$ denotes the weighted sequence space with the norm
\begin{eqnarray}\label{seqsp}
\|\{x_j\}_{j \ge 0}\|_{\ell_{p}^{a}}=
\begin{cases}
\ (\sum_{j \ge0} 2^{jap} |x_{j}|^{p})^{1/p} \quad \text{if}\quad p \neq \infty,\nonumber\\
\ \sup_{j \ge0} 2^{ja}|x_{j}| \quad \text{if}\quad p=\infty.\nonumber
\end{cases}
\end{eqnarray}
Note also that the above three estimates \eqref{bi1}, \eqref{bi2} and \eqref{bi3} are equivalent to
\begin{align}
 \left\|T (F,G)\right\|_{\ell_{\infty}^{\beta_{0}}(\mathbb{C})} &\leq C \|F\|_{L^{2}(|(x,t)|^{\alpha p})} \|G\|_{L^{2}(|(x,t)|^{\alpha p})},  \nonumber \\
\label{op_bi2}	 \left\|T (F,G)\right\|_{\ell_{\infty}^{\beta_{1}}(\mathbb{C})} &\leq C \|F\|_{L^{2}(|(x,t)|^{\alpha p})} \|G\|_{2}, \\
\label{op_bi3}	 \left\|T (F,G)\right\|_{\ell_{\infty}^{\beta_{1}}(\mathbb{C})} &\leq C \|F\|_{2} \|G\|_{L^{2}(|(x,t)|^{\alpha p})},
\end{align}
respectively, with $\beta_0 = -(\frac{n+2}{2} - \alpha p)$ and $\beta_1 = -(\frac{n+2}{2} - \frac{\alpha p}{2})$. Then, applying Lemma \ref{bil} with $\theta_0 = \theta_1 = 1/p'$ and $q=r=2$, we get for $1 < p < 2$,
\begin{equation*}
T : (L^2 (|(x,t)|^{\alpha p}), L^2)_{1/p', 2} \times (L^2 (|(x,t)|^{\alpha p}), L^2)_{1/p', 2} \rightarrow (\ell_{\infty}^{\beta_0}(\mathbb{C}), \ell_{\infty}^{\beta_1}(\mathbb{C}))_{2/p' , 1}.
\end{equation*}
Now, we use the following real interpolation space identities (see Theorems 5.4.1 and 5.6.1 in \cite{BL}):

\begin{lem}\label{rint}
	Let $0 < \theta <1$. Then one has
	\begin{equation*}
	(L^{2} (w_0), L^{2} (w_1))_{\theta,\,2} = L^2 (w), \quad w=w_{0} ^{1-\theta}w_{1} ^{\theta},
	\end{equation*}
	and for $1 \leq q_0, q_1, q \leq \infty$ and $ s_0 \ne s_1$,
	\begin{equation*}
	(\ell_{q_0}^{s_0}, \ell_{q_1}^{s_1})_{\theta, q} = \ell_{q}^{s}, \quad s=(1-\theta)s_0 + \theta s_1.
	\end{equation*}
\end{lem}

Then, for $1<p<2$, we have
\begin{equation*}
(L^2(|(x,t)|^{\alpha p}), L^2)_{1/p', 2} = L^2 (|(x,t)|^{\alpha})
\end{equation*}
and
\begin{equation*}
(\ell_{\infty}^{\beta_0}(\mathbb{C}), \ell_{\infty}^{\beta_1}(\mathbb{C}))_{2/p', 1} = \ell_{1}^{0}(\mathbb{C})
\end{equation*}
if $(1-\frac{2}{p'})\beta_0 + \frac{2}{p'}\beta_1 = 0$ (i.e., $ \alpha = \frac{n+2}{2}$).
Hence, we get \eqref{after_int} if $\alpha = (n+2)/2$ and $ 1 < p < \frac{n+1}{\alpha} (< 2)$.
When $ \alpha > (n+2)/2$, note that $\gamma : = \frac{p'}{2}(\alpha - \frac{n+2}{2}) > 0 $.
Since $j \ge 0$ and $\beta_1 <0$, the estimates \eqref{op_bi2} and \eqref{op_bi3} are trivially satisfied for $\beta_1$ replaced by $\beta_1 - \gamma$. Hence, by the same argument we only need to check that $(1-\frac{2}{p'})\beta_0 + \frac{2}{p'}(\beta_1 - \gamma) = 0$.
But this is an easy computation. Consequently, we get \eqref{after_int} if $\frac{n+2}{2}  \le \alpha < \frac{n+1}{p}$ for $1< p < \frac{2(n+1)}{n+2}$. This completes the proof.

Now it remains to show the estimates \eqref{bi1}, \eqref{bi2} and \eqref{bi3}. For $j \ge 0 $, let $\{Q_{\lambda}\}_{\lambda \in 2^j \mathbb{Z}^{n+1}}$ be a collection of cubes $Q_{\lambda} \subset \mathbb{R}^{n+1} $ centered at $\lambda$ with side length $2^j$.
Then by disjointness of cubes, we see that
\begin{align*}
\big| \big< (\psi_j K)* F , G \big> \big|
& \leq \sum_{\lambda_1,\lambda_2  \in 2^j \mathbb{Z}^{n+1}} \big| \big< (\psi_j K)* (F\chi_{Q_{\lambda_1}}) , G\chi_{ Q_{\lambda_2}} \big> \big| \\
& \leq \sum_{\lambda_1 \in 2^j \mathbb{Z}^{n+1}} \big| \big< (\psi_j K)* (F\chi_{Q_{\lambda_1}}) , G\chi_{\widetilde{Q}_{\lambda_1}} \big> \big|,
\end{align*}
where $\widetilde{Q}_{\lambda}$ denotes the cube with side length $2^{j+2}$ and the same center as $Q_{\lambda}$. By the Young and Cauchy-Schwarz inequalities, it follows that
\begin{align}\label{piece_prod}
\hspace{-4pt}\big| \big< (\psi_j K)* F , G \big> \big|
& \leq \sum_{\lambda_1 \in 2^j \mathbb{Z}^{n+1}} \| (\psi_j K)* (F\chi_{Q_{\lambda_1}}) \|_{\infty} \| G\chi_{ \widetilde{Q}_{\lambda_1}} \|_1 \nonumber\\
& \leq \sum_{\lambda_1 \in 2^j \mathbb{Z}^{n+1}} \| \psi_j K \|_{\infty} \| F\chi_{Q_{\lambda_1}} \|_{1} \| G\chi_{ \widetilde{Q}_{\lambda_1}} \|_1 \nonumber\\
& \leq \| \psi_j K\|_{\infty}  \Big( \sum_{\lambda_1 \in 2^j \mathbb{Z}^{n+1}} \| F\chi_{Q_{\lambda_1}}\|_{1}^2 \Big)^{\frac{1}{2}} \Big( \sum_{\lambda_1 \in 2^j \mathbb{Z}^{n+1}} \| G\chi_{ \widetilde{Q}_{\lambda_1}} \|_1^2 \Big)^{\frac{1}{2}}.
\end{align}
Now we need to bound the terms
 $$ \| (\psi_j K) \|_{\infty},  \quad  \sum_{\lambda_1 \in 2^j \mathbb{Z}^{n+1}} \| F\chi_{Q_{\lambda_1}}\|_{1}^2,  \quad \sum_{\lambda_1 \in 2^j \mathbb{Z}^{n+1}} \| G\chi_{ \widetilde{Q}_{\lambda_1}} \|_1^2. $$
For the first term, we use the stationary phase lemma, Lemma \ref{stationary_phase}, which is essentially due to Littman \cite{L}
(see also \cite{St}, Chap. VIII). Indeed, by applying the lemma with $\psi(\xi) = \sqrt{1+|\xi|^2}$, it follows that
\begin{equation*}
|K(x,t)| = \bigg| \int_{\mathbb{R}^{n}} e^{i(x \cdot \xi + t\sqrt{1+|\xi|^2})} |\xi|^{2\sigma}\varphi(\xi)^2 d\xi \bigg|
\leq C (1+|(x,t)|)^{-\frac{n}{2}}
\end{equation*}
since $\sigma\geq0$ and rank $H\psi =n$ for each $\xi \in \{ \xi \in \mathbb{R}^n : |\xi| \leq 2 \}$.
Thus we get
\begin{equation}\label{piece1}
\| \psi_j K \|_{\infty} \leq C 2^{-\frac{nj}{2}}.
\end{equation}

\begin{lem}\label{stationary_phase}
Let $H\psi$ be the Hessian matrix given by $(\frac{\partial^2 \psi}{\partial \xi_{i} \partial\xi_{j}})$.
Suppose that $\varphi$ is a compactly supported smooth function on $\mathbb{R}^n$ and $\psi$ is a smooth function which satisfies rank $H\psi \geq k$ on the support of $\varphi$. Then, for $(x,t) \in \mathbb{R}^{n+1}$,
	\begin{equation*}
	\bigg| \int_{\mathbb{R}^n} e^{i(x\cdot \xi + t \psi(\xi))} \varphi(\xi) d\xi \bigg| \leq C (1+ |(x,t)|)^{-\frac{k}{2}}.
	\end{equation*}
\end{lem}

For the second term, we have
\begin{align}\label{piece2}
\sum_{\lambda_1 \in 2^j \mathbb{Z}^{n+1}}& \| F \chi_{Q_{\lambda_1}} \|_1^2\nonumber\\
& = \sum_{\lambda_1 \in 2^j \mathbb{Z}^{n+1}} \Big( \int_{Q_{\lambda_1}} |F\chi_{Q_{\lambda_1}}| |(x,t)|^{\frac{\alpha p}{2}}|(x,t)|^{-\frac{\alpha p}{2}} dxdt \Big)^2 \nonumber\\
& \leq \sum_{\lambda_1 \in 2^j \mathbb{Z}^{n+1}} \Big( \int_{Q_{\lambda_1}} |F\chi_{Q_{\lambda_1}}|^2 |(x,t)|^{\alpha p}  dxdt \Big) \Big( \int_{Q_{\lambda_1}} |(x,t)|^{-\alpha p} dxdt \Big) \nonumber\\
& \leq  \sup_{\lambda_1 \in 2^j \mathbb{Z}^{n+1}} \Big( \int_{Q_{\lambda_1}} |(x,t)|^{-\alpha p} dxdt \Big) \sum_{\lambda_1 \in 2^j \mathbb{Z}^{n+1}} \Big( \int_{Q_{\lambda_1}} |F\chi_{Q_{\lambda_1}}|^2 |(x,t)|^{\alpha p}  dxdt \Big)\nonumber\\
&\leq C 2^{j(n+1 -\alpha p)} \| F \|_{L^2 (|(x,t)|^{\alpha p})}^2
\end{align}
whenever $0<\alpha p<n+1$, while
\begin{align}\label{piece3}
\sum_{\lambda_1 \in 2^j \mathbb{Z}^{n+1}} \| F \chi_{Q_{\lambda_1}} \|_1^2
& \leq \sum_{\lambda_1 \in 2^j \mathbb{Z}^{n+1}} \| F \chi_{Q_{\lambda_1}} \|_2^2 \| \chi_{Q_{\lambda_1}} \|_2^2 \nonumber\\
& \leq C2^{j(n+1)}\| F \|_2^2.
\end{align}
Similarly for $\sum_{\lambda_1 \in 2^j \mathbb{Z}^{n+1}} \| G \chi_{\widetilde{Q}_{\lambda_1}} \|_1^2$.
 Now, combining \eqref{piece_prod}, \eqref{piece1}, \eqref{piece2} and \eqref{piece3}, we obtain the desired estimates \eqref{bi1}, \eqref{bi2} and \eqref{bi3}.

\subsection{Proof of \eqref{LE2}}
The proof of the high frequency part \eqref{LE2} is much more delicate.
We generate further localizations in time and space to make \eqref{LE2} sharp as far as possible,
and need to analyze some relevant oscillatory integrals more carefully (see Lemma \ref{osci})
based on the stationary phase lemma, Lemma \ref{stationary_phase}.

By the $TT^*$ argument as before, we may show
\begin{equation*}
\bigg\|\int_{-\infty}^{\infty} |\nabla|^{2\sigma} e^{i(t-s)\sqrt{1-\Delta}}P_k^2 F(\cdot,s)\, ds\bigg\|_{L^2_{x,t}(|(x,t)|^{-\alpha})}
 \lesssim 2^{2k\sigma+\frac{kn}{2p}}  \|F\|_{L^2_{x,t}(|(x,t)|^{\alpha})}.
\end{equation*}
By dividing the integral $\int_{-\infty}^{\infty}$ into two parts $\int_{-\infty}^{t}$ and $\int_{t}^{\infty}$,
and then using duality, we are reduced to showing the following bilinear form estimate
\begin{align}\label{biest}
\bigg|\bigg\langle\int_{-\infty}^{t} |\nabla|^{2\sigma} e^{i(t-s)\sqrt{1-\Delta}}P_k^2 &F(\cdot,s)  ds, G(x,t)\bigg\rangle_{L_{x,t}^2}\bigg|\nonumber\\
&\lesssim 2^{2k\sigma+\frac{kn}{2p}} \|F\|_{L^2_{x,t}(|(x,t)|^{\alpha})}\|G\|_{L^2_{x,t}(|(x,t)|^{\alpha})}.
\end{align}
For this, we first decompose dyadically the left-hand side of \eqref{biest} in time; for fixed $j \ge 1$, define intervals $I_j = [2^{j-1}, 2^j)$ and $ I_0 = [0,1)$. Then we may write
\begin{equation}\label{jsum}
\bigg\langle \int_{-\infty}^{t} |\nabla|^{2\sigma} e^{i(t-s)\sqrt{1-\Delta}}P_k^2 F(\cdot,s)ds, G(x,t)\bigg\rangle_{L_{x,t}^2}
=\sum_{j \ge 0}B_j(F,G),
\end{equation}
where
$$B_j(F,G)=\int_{-\infty}^\infty\int_{t-I_j}\Big\langle|\nabla|^{2\sigma} e^{i(t-s)\sqrt{1-\Delta}}\,P_k^2 F(\cdot,s), G(x,t)\Big\rangle_{L_x^2} dsdt$$
with $t-I_j = (t-2^j, t-2^{j-1}]$ for $j\ge1$ and $t-I_0 = (t-1,t]$.
For these dyadic pieces $B_j$, we will obtain the following estimates in the next subsection.

\begin{prop}\label{local}
	Let $n \ge 1$. Then we have
	\begin{equation}\label{L2}
	|B_j(F,G)| \le C 2^{2k\sigma + j} \| F \|_{L^2} \| G \|_{L^2}
	\end{equation}
	and
	\begin{equation}\label{wL2}
	|B_j(F,G)|\le C 2^{2k\sigma+\frac{nk}2} 2^{j(\frac{n}{2} + 1-\alpha p)} \|F\|_{L^2 (|(x,t)|^{\alpha p})}\|G\|_{L^2 (|(x,t)|^{\alpha p})}
	\end{equation}
	for $\alpha < (n+1)/p$.
\end{prop}

Now we shall deduce the desired estimate \eqref{biest} from making use of the bilinear interpolation between \eqref{L2} and \eqref{wL2}.
First we define a bilinear vector-valued operator
$$ B(F,G) = \big\{ B_j(F,G) \big\}_{j \ge 0}.$$
Then, \eqref{L2} and \eqref{wL2} are equivalent to
$$ B : L^2 \times L^2 \rightarrow \ell_{\infty}^{s_0} \quad  \textnormal{and} \quad  B: L^2 (|(x,t)|^{\alpha p}) \times L^2 (|(x,t)|^{\alpha p}) \rightarrow \ell_{\infty}^{s_1}$$
with the operator norms $C2^{2k\sigma}$ and $C2^{2k\sigma+\frac{nk}2}$, respectively, where $s_0 = -1$ and $s_1 = -(n/2 + 1 -\alpha p)$.
By applying Lemma \ref{biint} below with $\theta = 1/p, \, q = \infty$ and $p_1 = p_2 = 2$,
we now obtain, for $1<p<\infty$ and $\alpha < (n+1)/p$,
\begin{equation*}
B: \big(L^2, L^2(|(x,t)|^{\alpha p})\big)_{1/p, 2} \times \big(L^2, L^2(|(x,t)|^{\alpha p})\big)_{1/p, 2}
\rightarrow \big(\ell_{\infty}^{s_0}, \ell_{\infty}^{s_1}\big)_{1/p, \infty}
\end{equation*}
with the operator norm $C2^{2k\sigma(1-1/p)}2^{(2k\sigma+\frac{nk}2)/p}=C2^{2k\sigma+\frac{kn}{2p}}$. 

\begin{lem}[\cite{BL}, Section 3.13, Exercise 5(a)]\label{biint}
	For $i=0,1$, let $A_i,B_i,C_i$ be Banach spaces and let $T$ be a bilinear operator such that
	$$T:A_0\times B_0\rightarrow C_0\quad\text{and}\quad T:A_1\times B_1\rightarrow C_1.$$
	Then one has, if\, $0<\theta<1$, $1\leq q \leq \infty$ and $1/q = 1/p_1 + 1/p_2 -1$,
	$$T:(A_0,A_1)_{\theta,p_1}\times(B_0,B_1)_{\theta,p_2} \rightarrow (C_0,C_1)_{\theta,q}.$$
\end{lem}

Finally, applying Lemma \ref{rint} with $\theta=1/p$,
we see that
$$ B: L^2(|(x,t)|^{\alpha}) \times L^2(|(x,t)|^{\alpha}) \rightarrow \ell_{\infty}^{s_2}$$
with the operator norm $C2^{2k\sigma+\frac{kn}{2p}}$
if $1<p<\infty$, $\alpha < (n+1)/p$, $n/2\neq\alpha p$ and
$$ s_2= (1- \frac{1}{p}) s_0 + \frac{1}{p} s_1 = -(1 + \frac{n}{2p} - \alpha).$$
This is equivalent to
\begin{equation*}
| B_j(F,G) | \lesssim  2^{2k\sigma+\frac{kn}{2p}} \,2^{j( 1 + \frac{n}{2p} - \alpha)} \| F \|_{L^2 (|(x,t)|^{\alpha})} \| G \|_{L^2 (|(x,t)|^{\alpha})}.
\end{equation*}
Summing this over $j \ge 0$ and using the decomposition \eqref{jsum}, the desired estimate \eqref{biest} follows when $ 1 + \frac{n}{2p} -\alpha<0 $.

In summary, all the requirements to obtain \eqref{LE2} are
$1<p<\infty$, $\alpha < (n+1)/p$, $n/2\neq\alpha p$, and $ 1 + \frac{n}{2p} -\alpha<0 $,
which are reduced to
\begin{equation*}
1 + \frac{n}{2p} < \alpha < \frac{n+1}{p}\quad\text{and}\quad 1<p<\frac{n+2}{2}.
\end{equation*}
But, this is weaker than the conditions $1+\frac n2 \le \alpha < \frac{n+1}p$ and $1<p<\frac{2(n+1)}{n+2}$ for which \eqref{LE} holds.
Therefore, we get \eqref{LE2} under the same conditions as in Proposition \ref{FLE}.

\subsubsection{Proof of Proposition \ref{local}}
Now we prove the estimates \eqref{L2} and \eqref{wL2} in Proposition \ref{local}.
Compared to the former that is relatively easy to prove, the proof of the latter needs further localizations in space
and to carefully analyze some relevant oscillatory integrals under the localization (see Lemma \ref{osci})
based on the stationary phase lemma, Lemma \ref{stationary_phase}.

\subsubsection*{Proof of \eqref{L2}}
For fixed $j \ge 0$, we first decompose $\mathbb{R}$ into intervals of length $2^j$ to have
\begin{equation}\label{decomt}
B_j(F,G)= \sum_{\ell \in \mathbb{Z}} \int_{2^j \ell}^{2^j (\ell+1)} \int_{t-I_j} \Big\langle |\nabla|^{2\sigma} e^{i(t-s) \sqrt{1-\Delta}}P_k^2 F(\cdot, s), G(x,t) \Big\rangle_{L_x^2} dsdt.
\end{equation}
By H\"{o}lder's inequality and Plancherel's theorem in $x$-variable, we then see that
\begin{align}\label{piece}
\int_{2^j \ell}^{2^j (\ell+1)} & \int_{t-I_j} \Big| \Big\langle |\nabla|^{2\sigma} e^{i(t-s) \sqrt{1-\Delta}} P_k^2 F(\cdot, s), G(x,t) \Big\rangle_{L_x^2} \Big| dsdt \nonumber\\
& \le \int_{2^j \ell}^{2^j (\ell+1)} \int_{t-I_j} \big\| |\nabla|^{2\sigma} e^{i(t-s) \sqrt{1-\Delta}} P_k^2 F(\cdot, s)\big\|_{L_x^2} \| G \|_{L_x^2} dsdt \nonumber\\
& \le 2^{2k\sigma} \int_{2^j \ell}^{2^j (\ell+1)} \int_{2^j (\ell-1)}^{2^j (\ell+ \frac{1}{2})} \| F \|_{L_y^2} \| G \|_{L_x^2} dsdt.
\end{align}
Using H\"{o}lder's inequality again and then applying the Cauchy-Schwarz inequality in $\ell$, we obtain
\begin{align}\label{lsum}
\sum_{\ell \in \mathbb{Z}}  \int_{2^j \ell}^{2^j (\ell+1)} &\int_{2^j (\ell-1)}^{2^j (\ell+ \frac{1}{2})} \| F \|_{L_y^2} \| G \|_{L_x^2} dsdt \nonumber\\
& \le C 2^j \sum_{\ell \in \mathbb{Z}} \Big( \int_{2^j (\ell-1)}^{2^j (\ell+ \frac{1}{2})}  \| F \|_{L_y^2}^2 ds \Big)^{\frac{1}{2}}
\Big(  \int_{2^j \ell}^{2^j (\ell+1)}\| G \|_{L_x^2}^2 dt \Big)^{\frac{1}{2}} \nonumber\\
& \le C 2^j \| F \|_{L^2} \| G \|_{L^2}.
\end{align}
Combining \eqref{decomt}, \eqref{piece} and \eqref{lsum} yields the first estimate \eqref{L2}.

\subsubsection*{Proof of \eqref{wL2}}
For fixed $j \ge 0$, we further decompose $\mathbb{R}^n$ into cubes of side length $2^j$ to see
\begin{equation}
F(y,s) = \sum_{\rho \in \mathbb{Z}^n} F_{\rho} (y,s) \quad \textnormal{and} \quad G(x,t) = \sum_{\rho \in \mathbb{Z}^n} G_{\rho} (x,t), \nonumber
\end{equation}
where $$ F_{\rho} (y,s) = \chi_{2^j\rho +[0, 2^j )^n} (y) F(y,s) \quad \textnormal{and} \quad G_{\rho} (x,t) = \chi_{2^j \rho +[0, 2^j )^n} (x) G(x,t). $$
By inserting this decomposition into \eqref{decomt}, we are reduced to showing that
\begin{align}\label{reddu}
\nonumber\sum_{\ell \in\mathbb{Z}}\sum_{\rho_1,\rho_2\in\mathbb{Z}^n} & \int_{2^j\ell}^{2^j(\ell+1)}\int_{t-I_j} \Big| \Big\langle|\nabla|^{2\sigma}e^{i(t-s)\sqrt{1-\Delta}}P_k^2 F_{\rho_1}(\cdot,s),G_{\rho_2}(x,t)\Big\rangle_{L_x^2}\Big|dsdt \notag\\
&\lesssim 2^{2k\sigma+\frac{nk}2}  2^{j(\frac{n}{2}+1 - \alpha p)}\|F\|_{L_{y,s}^2(|(y,s)|^{\alpha p})}\|G\|_{L_{x,t}^2({|(x,t)|^{\alpha p}})}
\end{align}
for $\alpha < (n+1)/p$.
To show this, we first write
\begin{align}\label{PH}
|\nabla|^{2\sigma}&e^{i(t-s)\sqrt{1-\Delta}}P_k^2 F_{\rho_1}(\cdot,s)\nonumber\\
&=\int_{ \mathbb{R}^n}\bigg(\int_{\mathbb{R}^n} e^{i(x-y)\cdot\xi+i(t-s)\sqrt{1+|\xi|^2}} |\xi|^{2\sigma}\phi(2^{-k}|\xi|)^2 d\xi\bigg)
F_{\rho_1}(y,s) dy\nonumber\\
&:=\int_{ \mathbb{R}^n}I_k(x-y,t-s)F_{\rho_1}(y,s) dy
\end{align}
and then obtain the following estimates:

\begin{lem}\label{osci}
Let $x\in 2^j \rho_1 +[0,2^j)^n$, $y\in 2^j\rho_2+[0,2^j)^n$ and $s\in t-I_j$.
Then,
\begin{equation}\label{nonSP}
\big|I_k(x-y,t-s)\big| \lesssim2^{2k\sigma}2^{kn}(2^{k+j} |\rho_1-\rho_2|)^{-10n}
\end{equation}
when $|\rho_1-\rho_2|\ge 4\sqrt{n}$, and when $|\rho_1-\rho_2|< 4\sqrt{n}$
\begin{equation}\label{SP}
\big| I_k(x-y,t-s)\big| \lesssim2^{2k\sigma+\frac{kn}2}2^{-\frac{n}{2}j}.
\end{equation}
\end{lem}

\begin{proof}[Proof of Lemma \ref{osci}]
To show \eqref{nonSP},
 we change variables $2^{-k}\xi\rightarrow\xi$, and then set $\psi(\xi):=(x-y)\cdot\xi+(t-s)\sqrt{1+|\xi|^2}$ and $\varphi(\xi):=|\xi|^{2\sigma}\phi(|\xi|)^2$. Then
 $$I_k(x-y,t-s)=2^{2k\sigma}2^{kn}\int_{\mathbb{R}^n} e^{i\psi(2^k\xi)}\varphi(\xi)d\xi,$$
and when $|x-y|\geq 2|t-s|$ we see that
\begin{equation*}
\bigg|\int_{\mathbb{R}^n} e^{i\psi(2^k\xi)}\varphi(\xi)d\xi\bigg|
\lesssim(2^{k} |x-y|)^{-10n}
\end{equation*}
by using integration by parts repeatedly.
Now we note that $|t-s|\geq2^j$, and
$|2^{-j}x-2^{-j}y|\geq |\rho_1 -\rho_2|/2$ when $|\rho_1-\rho_2|> \sqrt{n}$,
if $x\in 2^j \rho_1 +[0,2^j)^n$, $y\in 2^j\rho_2+[0,2^j)^n$ and $s\in t-I_j$.
Therefore, if $|\rho_1-\rho_2|\ge 4\sqrt{n}$, we get \eqref{nonSP} since
$|x-y|\geq 2^{j-1}|\rho_1 -\rho_2| \geq 2\sqrt{n}|t-s|$.
For the second estimate \eqref{SP} we use the stationary phase lemma, Lemma \ref{stationary_phase}, as before.
Then
\begin{align*}
\int_{\mathbb{R}^n} e^{i\psi(2^k\xi)}\varphi(\xi)d\xi
&=\int_{\mathbb{R}^n} e^{i2^k(x-y)\cdot\xi+i2^k(t-s)\sqrt{2^{-2k}+|\xi|^2}}\varphi(\xi) d\xi\\
&\leq C (1+2^k|(x-y,t-s)|)^{-\frac{n}{2}}\\
&\leq C2^{-\frac{n}{2}(k+j)}
\end{align*}
since $0<2^{-2k}\leq1$ and the Hessian matrix of $\sqrt{2^{-2k}+|\xi|^2}$ has $n$ non-zero eigenvalues
for each $k\geq0$ and $\xi \in \{ \xi \in \mathbb{R}^n : |\xi| \sim1 \} $.
Here the constant $C$ occurring in Lemma \ref{stationary_phase} with $\psi=\sqrt{2^{-2k}+|\xi|^2}$
is independent of $k$. This is because the constant in the lemma depends only on upper and lower bounds for
finitely many derivatives of $\psi$ and $\varphi$ which are uniform in our case when $k\geq0$.
Therefore, we get \eqref{SP}.
\end{proof}

Now we return to the proof of \eqref{reddu} with what we just obtained.
Let us first denote
\begin{align*}
C_{\rho_1, \rho_2}(j)=
\begin{cases}
(2^j|\rho_1 -\rho_2|)^{-10n} \quad \text{if} \quad |\rho_1-\rho_2|\geq 4\sqrt{n},\\
2^{-nj/2} \quad \text{if} \quad |\rho_1 - \rho_2| < 4\sqrt{n},
\end{cases}
\end{align*}
and $C_{k,\sigma}=2^{2k\sigma+nk/2}$.
Using \eqref{PH}, \eqref{nonSP} and \eqref{SP}, we then have
\begin{align}\label{AfterPH}
\int_{2^j\ell}^{2^j(\ell+1)} &\int_{t-I_j} \Big|\Big<|\nabla|^{2\sigma}e^{i(t-s)\sqrt{1-\Delta}}P_k^2 F_{\rho_1}(\cdot,s),G_{\rho_2}(x,t)\Big>_{L_x^2}\Big|\,dsdt\notag\\
\lesssim &\,C_{k,\sigma}C_{\rho_1,\rho_2}(j) \int_{2^j\ell}^{2^j(\ell+1)} \int_{t-I_j} \int_{\mathbb{R}^n}\int_{\mathbb{R}^n} |F_{\rho_1}(y,s)G_{\rho_2}(x,t)|\,dydxdsdt\notag\\
\lesssim &\, C_{k,\sigma}C_{\rho_1,\rho_2}(j) \int_{2^j\ell}^{2^j(\ell+1)}\int_{2^j(\ell-1)}^{2^j(\ell+\frac{1}{2})} \|F_{\rho_1}(\cdot,s)\|_{L_y^1}\|G_{\rho_2}(\cdot,t)\|_{L_x^1}\,dsdt.
\end{align}
Notice here that
\begin{align*}
& \int_{2^j(\ell-1)}^{2^j(\ell+\frac{1}{2})} \|F_{\rho_1}(\cdot,s)\|_{L_y^1}ds\\
&\quad = \int_{2^j(\ell-1)}^{2^j(\ell+\frac{1}{2})} \int_{y\in 2^j\rho_1+[0,2^j)^n} |F_{\rho_1}(y,s)|\cdot |(y,s)|^{\alpha p/2}\cdot|(y,s)|^{-\alpha p/2} dyds\\
& \quad \le \|F_{\rho_1}\chi_{[2^j(\ell-1),2^j(\ell+\frac{1}{2}))}\|_{L_{y,s}^{2}(|(y,s)|^{\alpha p})}\Big(\int_{2^j(\ell-1)}^{2^j(\ell+\frac{1}{2})}\int_{y\in 2^j\rho_1+[0,2^j)^n}|(y,s)|^{-\alpha p}\,dyds\Big)^{\frac{1}{2}}\\
&\quad\lesssim 2^{j(n+1-\alpha p)/2}\|F_{\rho_1}\chi_{[2^j(\ell-1),2^j(\ell+\frac{1}{2}))}\|_{L_{y,s}^{2}(|(y,s)|^{\alpha p})}
\end{align*}
if $\alpha < (n+1)/p$.
Similar calculation gives
\begin{equation*}
\int_{2^j\ell}^{2^j(\ell+1)}\|G_{\rho_2}(\cdot,t)\|_{L_x^1}dt \lesssim 2^{j(n+1-\alpha p)/2}\|G_{\rho_2}\chi_{[2^j\ell,2^j(\ell+1))}\|_{L_{x,t}^{2}(|(x,t)|^{\alpha p})}.
\end{equation*}
By applying these estimates to \eqref{AfterPH} and then using the Cauchy-Schwarz inequality in $\ell$, we now see that
\begin{align}\label{AfterPH2}
&\sum_{\ell\in\mathbb{Z}}\sum_{\rho_1,\rho_2\in\mathbb{Z}^n}\int_{2^j\ell}^{2^j(\ell+1)}\int_{t-I_j}\Big|\Big<|\nabla|^{2\sigma}e^{i(t-s)\sqrt{1-\Delta}}P_k^2 F_{\rho_1}(\cdot,s),G_{\rho_2}(x,t)\Big>_{L_x^2}\Big|\,dsdt\notag\\
& \quad \lesssim C_{k,\sigma}2^{j(n+1-\alpha p)}\sum_{\rho_1,\rho_2\in\mathbb{Z}^n}C_{\rho_1,\rho_2}(j)\|F_{\rho_1}\|_{L_{y,s}^2(|(y,s)|^{\alpha p})}\|G_{\rho_2}\|_{L_{x,t}^2(|(x,t)|^{\alpha p})}.
\end{align}
When $|\rho_1 - \rho_2|<4\sqrt{n}$, the sum in the right-hand side of \eqref{AfterPH2} is bounded by the Cauchy-Schwarz inequality in $\rho_1$,
as follows:
\begin{align*}
&\sum_{\{\rho_1,\rho_2:|\rho_1-\rho_2|<4\sqrt{n}\}}C_{\rho_1,\rho_2}(j)\|F_{\rho_1}\|_{L_{y,s}^2(|(y,s)|^{\alpha p})}\|G_{\rho_2}\|_{L_{x,t}^2(|(x,t)|^{\alpha p})}\\
& \quad\le 2^{-\frac{nj}{2}} \Big(\sum_{\rho_1} \|F_{\rho_1}\|_{L_{y,s}^2(|(y,s)|^{\alpha p})}^2\Big)^{\frac{1}{2}}\Big(\sum_{\rho_1}\Big(\sum_{\{\rho_2:|\rho_1-\rho_2|<4\sqrt{n}\}}\|G_{\rho_2}\|_{L_{x,t}^2(|(x,t)|^{\alpha p})}\Big)^2\Big)^{\frac{1}{2}}.
\end{align*}
Since the number of $\rho_2$ satisfying $|\rho_1 - \rho_2|<4\sqrt{n}$ for fixed $\rho_1$ is at most finite,
the right-hand side in the above inequality is bounded by
\begin{equation*}
C2^{-\frac{nj}{2}} \|F\|_{L_{y,s}^2(|(y,s)|^{\alpha p})}\|G\|_{L_{x,t}^2({|(x,t)|^{\alpha p}})}.
\end{equation*}
Combining this and \eqref{AfterPH2} implies the desired estimate \eqref{reddu} in this case.
For the case where $|\rho_1 - \rho_2| \ge 4\sqrt{n}$, we divide cases into two parts as
\begin{align*}
\sum_{\{\rho_1,\rho_2:|\rho_1-\rho_2|\ge 4\sqrt{n}\}}&C_{\rho_1,\rho_2}(j)\|F_{\rho_1}\|_{L_{y,s}^2(|(y,s)|^{\alpha p})}\|G_{\rho_2}\|_{L_{x,t}^2(|(x,t)|^{\alpha p})}\\
&=\sum_{\{\rho_1,\rho_2:|\rho_1-\rho_2|\ge 4\sqrt{n},|\rho_1|\ge 1 \}} + \sum_{\{\rho_1,\rho_2:|\rho_1-\rho_2|\ge 4\sqrt{n},|\rho_1|<1 \}}.
\end{align*}
By H\"{o}lder's inequality in $\rho_2$ and then Young's inequality, the first term is bounded as
\begin{align*}
2^{-10nj} &\Big\|\sum_{\{\rho_1:|\rho_1|\geq 1\}}(|\rho_1 - \rho_2|)^{-10n} \|F_{\rho_1}\|_{L_{y,s}^2(|(y,s)|^{\alpha p})}\Big\|_{l^2} \Big(\sum_{\rho_2\in\mathbb{Z}^n} \|G_{\rho_2}\|_{L_{x,t}^2(|(x,t)|^{\alpha p})}\Big)^{\frac{1}{2}}\notag\\
&\le 2^{-10nj}\Big(\sum_{\{\rho_1:|\rho_1|\geq 1\}} |\rho_1|^{-10n}\Big)\Big(\sum_{\rho_1\in\mathbb{Z}^n}\|F_{\rho_1}\|_{L_{y,s}^2(|(y,s)|^{\alpha p})}\Big)^{\frac{1}{2}} \|G\|_{L_{x,t}^2({|(x,t)|^{\alpha p}})}\notag\\
&\le C2^{-10nj}\|F\|_{L_{y,s}^2(|(y,s)|^{\alpha p})}\|G\|_{L_{x,t}^2({|(x,t)|^{\alpha p}})},
\end{align*}
while the second term is bounded as
\begin{align*}
&\sum_{\{\rho_2:|\rho_2|\geq 4\sqrt{n}\}}(2^{j}|\rho_2|)^{-10n}\|F_0\|_{L_{y,s}^2(|(y,s)|^{\alpha p})}\|G_{\rho_2}\|_{L_{x,t}^2(|(x,t)|^{\alpha p})}\notag\\
&\qquad\le 2^{-10nj}\|F_0\|_{L_{y,s}^2(|(y,s)|^{\alpha p})} \Big(\sum_{\{\rho_2:|\rho_2|\geq 1\}} |\rho_2|^{-10n}\Big)^{\frac{1}{2}}\Big(\sum_{\rho_2\in\mathbb{Z}^n}\|G_{\rho_2}\|_{L_{x,t}^2(|(x,t)|^{\alpha p})}\Big)^{\frac{1}{2}}\notag\\
&\qquad\le C2^{-10nj}\|F\|_{L_{y,s}^2(|(y,s)|^{\alpha p})}\|G\|_{L_{x,t}^2({|(x,t)|^{\alpha p}})}.
\end{align*}
Therefore, we get \eqref{reddu} in this case as well.

\end{document}